\documentclass{article}

\usepackage{graphicx}
\usepackage{color}
\usepackage{stmaryrd}
\usepackage[all]{xy}

\usepackage[numbers]{natbib}
\usepackage{natbibspacing}
\renewcommand{\bibfont}{\small}

\definecolor{myurlcolor}{rgb}{0,0,0.5}
\usepackage{hyperref}
\hypersetup{colorlinks,
linkcolor=black,
citecolor=black,
urlcolor=myurlcolor}

\usepackage{bbm}
\usepackage[text={125mm,202mm},centering]{geometry}

\usepackage{tocloft}
\usepackage{latexsym}
\usepackage{amssymb,amsmath}
\newcommand{\hyph}{\mbox{-}}
\newcommand{\bref}[1]{(\ref{#1})}
\newcommand{\fcat}[1]{\mathbf{#1}}
\newcommand{\such}{:}
\DeclareMathOperator{\Hom}{Hom}
\newcommand{\id}{\mathrm{id}}
\newcommand{\Set}{\fcat{Set}}
\newcommand{\Mod}{\mathbf{Mod}}
\newcommand{\demph}[1]{\textbf{\textup{#1}}}
\newcommand{\iso}{\cong}
\newcommand{\sub}{\subseteq}
\renewcommand{\to}{\rightarrow}
\newcommand{\toby}[1]{\stackrel{#1}{\rightarrow}}
\newcommand{\epic}{\twoheadrightarrow}
\DeclareMathOperator{\im}{im}
\newcommand{\from}{\colon}
\newcommand{\utm}[3]{\bigl(
\begin{smallmatrix}
#1      &#2     \\
0       &#3
\end{smallmatrix}
\bigr)}
\newcommand{\utmd}[3]{%
\begin{pmatrix}
#1      &#2     \\
0       &#3
\end{pmatrix}}
\DeclareMathOperator{\End}{End}
\DeclareMathOperator{\rad}{rad}
\newcommand{\qiso}{\mathord{\cong}}
\newcommand{\of}{\mathbin{\circ}}

\usepackage[amsmath,thmmarks]{ntheorem}

\newtheorem{thm}{Theorem}[section]
\newtheorem{propn}[thm]{Proposition}
\newtheorem{lemma}[thm]{Lemma}
\newtheorem{cor}[thm]{Corollary}

\theorembodyfont{\normalfont}

\newtheorem{example}[thm]{Example}

\theoremstyle{nonumberplain}
\theoremsymbol{\ensuremath{\square}}
\qedsymbol{\ensuremath{\square}}

\newtheorem{proof}{Proof}

\newcommand{\theoremtobeproved}{}
\newtheorem{pfoftheorem}{Proof of \theoremtobeproved}

\author{Tom Leinster%
\thanks{School of Mathematics, University of Edinburgh.
Tom.Leinster@ed.ac.uk}}
\title{The bijection between\\ projective indecomposable and simple
  modules}
\date{\vspace{-5ex}}

\begin{document}

\sloppy
\maketitle

\begin{abstract}  
For modules over a finite-dimensional algebra, there is a canonical
one-to-one correspondence between the projective indecomposable modules and
the simple modules.  In this purely expository note, we take a
straight-line path from the definitions to this correspondence.  The proof
is self-contained.
\end{abstract}

\tableofcontents

\section{Introduction}

One of the remarkable features of the representation theory of
finite-dimensional algebras $A$ is the existence of a canonical bijection
\[
\{ \text{projective indecomposable $A$-modules} \}/\qiso
\quad \longleftrightarrow \quad
\{ \text{simple $A$-modules} \}/\qiso
\]
between the isomorphism classes of projective indecomposable modules and
the isomorphism classes of simple modules.  The bijection is given by
matching a projective indecomposable module $P$ with a simple module $S$
just when $S$ is a quotient of $P$.

The modest purpose of this expository note is to prove this correspondence,
starting from nothing.  Everything here is classical; nothing here is new.
For instance, almost all of what follows can be found in Chapter~1 of
Benson's book~\cite{Bens} or Chapter~I of Skowro\'nski and Yamagata's
book~\cite{SkYa}.

The exposition emphasizes the fact that the bijection can be established
without calling on any major theorems (or taking a detour to prove them).
This is certainly clear to many algebraists, but may not always be apparent
to the amateur.

In the final section, we do allow ourselves to use the Jordan--H\"older
theorem or the Krull--Schmidt theorem (either will do), but only to show
that the two sets of isomorphism classes related by the bijection are
finite.  The bijection itself is established without it.

Two features of the exposition are worth highlighting.  The first is the
indispensable role played by Fitting's lemma, and especially its corollary
that every endomorphism of an indecomposable module is either nilpotent or
invertible (Corollary~\ref{cor:indec-fitting}).  The second is the
observation that every projective indecomposable module is finitely
generated (a consequence of Lemma~\ref{lemma:pi-cyclic}).  Although this is
again well-known, it is perhaps not quite as well-known as it could be.

\section{Basic definitions}
\label{sec:basics}

Throughout, we fix a field $K$, not necessarily algebraically closed.  We
also fix a finite-dimensional $K$-algebra $A$.  Algebras are always taken
to be unital, but need not be commutative.

Terms such as vector space, linear, and dimension will always mean vector
space etc.\ over $K$.  Module will mean left $A$-module (not
necessarily finitely generated), and homomorphisms, endomorphisms and
quotients are understood to be of modules over $A$.  A double-headed arrow
$\epic$ denotes an epimorphism (surjective homomorphism).

Since $A$ is finite-dimensional, a module is finitely generated over $A$ if
and only if it is finite-dimensional over $K$.

A module is \demph{cyclic} if it is generated over $A$ by a single element,
or equivalently if it is a quotient of the $A$-module $A$.  It is
\demph{simple} if it is nonzero and has no nontrivial submodules.  It is
\demph{indecomposable} if it is nonzero and has no nontrivial direct
summands.

A module $P$ is \demph{projective} if the functor $\Hom_A(P, -): A\hyph\Mod
\to \Set$ preserves epimorphisms, in the categorical sense.  Concretely,
this means that given module homomorphisms $\phi$ and $\pi$ as shown, with
$\pi$ surjective, there exists a homomorphism $\psi$ making the triangle
commute:
\[
\xymatrix{
        &P \ar@{.>}[dl]_\psi \ar[dr]^\phi       &       \\
M \ar@{->>}[rr]_\pi &                           &N.
}
\]

In an intuitive sense, simple modules are `atomic', having no nontrivial
submodules.  Mere indecomposability is a weaker condition, but the
\emph{projective} indecomposable modules also have a claim to being the
`atoms' of $A$: for the $A$-module $A$ is isomorphic to a finite direct sum
of projective indecomposable modules, and every projective indecomposable
module appears as one of those summands.  (See Section~\ref{sec:ks}.)  These
two types of `atomic' module are genuinely different, since simple does not
imply projective indecomposable, nor vice versa (Example~\ref{eg:ut}).
Nonetheless, the canonical bijection that is the subject of this note
creates an intimate relationship between them.

We record some basic facts about projective modules.

\begin{lemma}
\label{lemma:proj-basic}
\begin{enumerate}
\item
\label{part:proj-summand}
Every direct summand of a projective module is projective.

\item
\label{part:proj-summand-free}
Every direct summand of a free module is projective.

\item
\label{part:proj-epi}
Let $M$ be a module and $P$ a projective module.  If $P$ is a quotient of
$M$ then $P$ is a direct summand of $M$.

\item
\label{part:proj-in-free}
Every projective module is a direct summand of a free module.
\end{enumerate}
\end{lemma}

\begin{proof}
For~\bref{part:proj-summand}, let $P$ be a projective module decomposed as
a direct sum $P \iso X \oplus Y$.  We show that $X$ is projective.  Write
$\sigma\from P \epic X$ and $\iota\from X \to P$ for the projection and
inclusion of the direct sum.  Take homomorphisms $\pi\from M \epic N$ and
$\phi\from X \to N$.  Since $P$ is projective, there exists a homomorphism
$\psi$ such that
\[
\xymatrix{
P \ar[r]^\sigma \ar@{.>}[d]_\psi&X \ar[dr]^\phi &       \\
M \ar@{->>}[rr]_\pi             &               &N
}
\]
commutes.  We now have a homomorphism $\psi\iota \from X \to M$, and $\pi
\psi \iota = \phi \sigma \iota = \phi$.

To deduce~\bref{part:proj-summand-free} from~\bref{part:proj-summand}, it is
enough to show that free modules are projective.  Let $F$ be free with
basis $(e_s)_{s \in S}$.  Take $\pi$ and $\phi$ as shown:
\[
\xymatrix{
        &F \ar@{.>}[dl]_\psi \ar[dr]^\phi       &       \\
M \ar@{->>}[rr]_\pi &                           &N.
}
\]
We may choose for each $s \in S$ an element $m_s \in M$ such that $\pi(m_s)
= \phi(e_s)$.  There is a unique $\psi\from F \to M$ such that $\psi(e_s) =
m_s$ for all $s$, and then $\pi\psi = \phi$.

For~\bref{part:proj-epi}, given $\pi\from M \epic P$, we may choose a
homomorphism $\iota$ such that
\[
\xymatrix{
        &P \ar@{.>}[dl]_\iota \ar[dr]^{\id_P}   &       \\
M \ar@{->>}[rr]_\pi &                           &P
}
\]
commutes.  An easy calculation shows that $M = \ker\pi \oplus \im \iota$;
but $\im \iota \iso P$, so $P$ is a direct summand of $M$.

Part~\bref{part:proj-in-free} follows, since every module is a quotient of
some free module.
\end{proof}

\section{Fitting's lemma}

Here we recall a very useful basic result about the dynamics of a linear
operator on a finite-dimensional vector space.

\begin{lemma}[Fitting]
\label{lemma:fitting}
Let $\theta$ be a linear endomorphism of a finite-dimensional vector space
$X$.  Then $X = \ker(\theta^n) \oplus \im(\theta^n)$ for all $n \gg 0$.
\end{lemma}

\begin{proof}
The chain $\ker(\theta^0) \sub \ker(\theta^1) \sub \cdots$ of linear
subspaces of $X$ must eventually stabilize, say at $\ker(\theta^m)$.  Let
$n \geq m$.  If $x \in \ker(\theta^n) \cap \im(\theta^n)$ then $x =
\theta^n(y)$ for some $y \in X$; but then $0 = \theta^n(x) =
\theta^{2n}(y)$, so $y \in \ker(\theta^{2n}) = \ker(\theta^n)$, so $x = 0$.
Hence $\ker(\theta^n) \cap \im(\theta^n) = 0$.  Since $\dim\ker(\theta^n) +
\dim\im(\theta^n) = \dim X$, the result follows.
\end{proof}

\begin{cor}
\label{cor:indec-fitting}
Every endomorphism of a  finitely generated indecomposable module is
either nilpotent or invertible.
\end{cor}

\begin{proof}
Let $\theta$ be an endomorphism of a finitely generated indecomposable
module $M$.  By Lemma~\ref{lemma:fitting}, we can choose $n \geq 1$ such
that $\ker(\theta^n) \oplus \im(\theta^n) = M$.  Since $M$ is
indecomposable, $\ker(\theta^n)$ is either $0$ or $M$.  If $0$ then
$\theta^n$ is injective, so $\theta$ is injective; but $\theta$ is a linear
endomorphism of a finite-dimensional vector space, so $\theta$ is
invertible.  If $M$ then $\theta^n = 0$, so $\theta$ is nilpotent. 
\end{proof}

\section{Maximal submodules}
\label{sec:max}

We will need to know that every projective indecomposable $A$-module has a
maximal (proper) submodule.  A simple application of Zorn's lemma does not
prove this, since the union of a chain of proper submodules need not be
proper.  (And in fact, not every module over every ring does have a maximal
submodule.)

To prove it, we use two constructions.  Let $M$ be a module.  We write
$\rad(M)$ for the intersection of all the maximal submodules of $M$ (the
\demph{Jacobson radical}).  Given a left ideal $I$ of $A$, we write $IM$
for the submodule of $M$ generated by $\{im \such i \in I, m \in M\}$.
Both constructions are functorial:

\begin{lemma}
\label{lemma:basic-functorial}
Let $f\from M \to N$ be a homomorphism of modules.  Then $f\rad(M) \sub
\rad(N)$ and $f(I M) \sub I N$, for any left ideal $I$ of $A$.
\end{lemma}

\begin{proof}
The second statement is trivial.  For the first, let $K$ be a maximal
submodule of $N$; we must prove that $f\rad(M) \sub K$.  Since $N/K$ is
simple, the image of the composite $M \toby{f} N \epic N/K$ is either $N/K$
or $0$, so the kernel is either maximal or $M$.  In either case, the kernel
contains $\rad(M)$, so $f\rad(M) \sub K$.
\end{proof}

\begin{lemma}
\label{lemma:rad-rad}
Let $M$ be a module.  Then $\rad(A)M \sub \rad(M)$, with equality if $M$ is
projective. 
\end{lemma}

\begin{proof}
For each $m \in M$, right multiplication by $m$ defines a homomorphism $A
\to M$, so $\rad(A)m \sub \rad(M)$ by Lemma~\ref{lemma:basic-functorial}.
This proves the inclusion.

Next we prove that the inclusion is an equality for free modules.  Let $F$
be free with basis $(e_s)_{s \in S}$.  Let $x = \sum_{s \in S} x_s e_s \in
\rad(F)$ (with $x_s = 0$ for all but finitely many $s$).  Applying
Lemma~\ref{lemma:basic-functorial} to the $s$-projection $F \to A$ gives
$x_s \in \rad(A)$, for each $s \in S$.  Hence $x \in \rad(A)F$, giving
$\rad(F) \sub \rad(A)F$.

Now let $P$ be any projective module.  By
Lemma~\ref{lemma:proj-basic}\bref{part:proj-in-free}, there is an
epimorphism $\pi\from F \epic P$ with a section $\iota\from P \to F$, for
some free $F$.  So
\[
\rad(P) = \pi\iota\rad(P) \sub \pi\rad(F) = \pi(\rad(A)F) \sub \rad(A)P,
\]
using Lemma~\ref{lemma:basic-functorial} twice.
\end{proof}

\begin{lemma}
\label{lemma:rad-nil}
$\rad(A)^n = 0$ for some $n \geq 0$.
\end{lemma}

\begin{proof}
Since $A$ is finite-dimensional, we can choose $n \geq 0$ minimizing the
dimension of the $A$-module $\rad(A)^n$.  Suppose that it is nonzero.  By
finite-dimensionality again, $\rad(A)^n$ has a maximal submodule, so
$\rad(\rad(A)^n)$ is a proper submodule of $\rad(A)^n$.  But $\rad(A)^{n +
  1} \sub \rad(\rad(A)^n)$ by Lemma~\ref{lemma:rad-rad}, so $\rad(A)^{n +
  1}$ is a proper submodule of $\rad(A)^n$, a contradiction.
\end{proof}

\begin{propn}
\label{propn:proj-max}
Every nonzero projective module has a maximal submodule.
\end{propn}

\begin{proof}
Let $P$ be a projective module with no maximal submodule.  Then $P =
\rad(P) = \rad(A)P$ by Lemma~\ref{lemma:rad-rad}, so $P = \rad(A)^n P$ for
all $n \geq 0$, so $P = 0$ by Lemma~\ref{lemma:rad-nil}.
\end{proof}

\section{From simple modules to projective indecomposable modules}

Here we show that every simple module $S$ has associated with it a unique
projective indecomposable module $P$, characterized by the existence of an
epimorphism $P \epic S$.

Roughly, our first lemma says that different projective indecomposable
modules share no quotients.

\begin{lemma}
\label{lemma:no-share}
Let $P$ and $P'$ be projective indecomposable modules, at least one of
which is finitely generated.  If some nonzero module is a quotient of
both $P$ and $P'$ then $P \iso P'$.
\end{lemma}

\begin{proof}
Suppose that $P$ is finitely generated, and that there exist a nonzero
module $M$ and epimorphisms $\pi\from P \epic M$, $\pi'\from P' \epic M$.
Since $P$ and $P'$ are projective, there exist homomorphisms
\[
\xymatrix{
P \ar@{->>}[rd]_{\pi} \ar@<.5ex>[rr]^{\alpha'} &       &
P' \ar@{->>}[ld]^{\pi'} \ar@<.5ex>[ll]^{\alpha}       \\
                &M
}
\]
such that $\pi'\alpha' = \pi$ and $\pi\alpha = \pi'$.  Then $\alpha\alpha'$
is an endomorphism of $P$ satisfying $\pi(\alpha\alpha') = \pi$.  Since $P$
is indecomposable, Corollary~\ref{cor:indec-fitting} implies that
$\alpha\alpha'$ is nilpotent or invertible.  If nilpotent then
$(\alpha\alpha')^n = 0$ for some $n \geq 0$, so $\pi = \pi(\alpha\alpha')^n
= \pi 0 = 0$, contradicting the fact that $\pi$ is an epimorphism to a
nonzero module.  So $\alpha\alpha'$ is invertible, and in particular
$\alpha$ is an epimorphism.  By
Lemma~\ref{lemma:proj-basic}\bref{part:proj-epi}, $P$ is therefore a direct
summand of $P'$.  But $P'$ is indecomposable and $P$ is nonzero, so $P
\iso P'$.
\end{proof}

\begin{lemma}
\label{lemma:simple-cyclic}
Every simple module is cyclic.
\end{lemma}

\begin{proof}
Let $S$ be a simple module.  Since $S$ is nonzero, we may choose a nonzero
element $x \in S$.  The submodule generated by $x$ is nonzero, and is
therefore $S$.
\end{proof}

\begin{lemma}
\label{lemma:simple-some-qt}
Every simple module is a quotient of some cyclic projective indecomposable
module.
\end{lemma}

\begin{proof}
Let $S$ be a simple module.  By Lemma~\ref{lemma:simple-cyclic}, $S$ is a
quotient of the $A$-module $A$.  Thus, among all direct summands $M$ of $A$
with the property that $S$ is a quotient of $M$, we may choose one of
smallest dimension; call it $P$.  Then $P$ is projective (by
Lemma~\ref{lemma:proj-basic}\bref{part:proj-summand-free}) and cyclic (being a
quotient of $A$).  To see that $P$ is indecomposable, suppose that $P = M
\oplus N$ for some submodules $M$ and $N$.  Take an epimorphism $\pi\from P
\epic S$.  Then $S = \pi M + \pi N$ and $S$ is nonzero, so without loss of
generality, $\pi M$ is nonzero.  Since $S$ is simple, $\pi M = S$, so $S$
is a quotient of $M$.  But $M$ is a direct summand of $A$, so minimality of
$P$ gives $M = P$, as required.
\end{proof}

The next result can be compared to Lemma~\ref{lemma:simple-cyclic}.  It
implies, in particular, that projective indecomposable modules are finitely
generated. 

\begin{lemma}
\label{lemma:pi-cyclic}
Every projective indecomposable module is cyclic.
\end{lemma}

\begin{proof}
Let $P$ be a projective indecomposable module.  By
Proposition~\ref{propn:proj-max}, we may choose a maximal submodule of $P$,
the quotient by which is a simple module: say $\pi \from P \epic S$.  By
Lemma~\ref{lemma:simple-some-qt}, we may then choose $\pi' \from P' \epic
S$ with $P'$ cyclic and projective indecomposable.  By
Lemma~\ref{lemma:no-share}, $P \iso P'$.  Hence $P$ is cyclic.
\end{proof}

\begin{propn}
\label{propn:proj-cover}
For each simple module $S$, there is a projective indecomposable module,
unique up to isomorphism, of which $S$ is a quotient.
\end{propn}

\begin{proof}
Lemma~\ref{lemma:simple-some-qt} proves existence.
Lemmas~\ref{lemma:no-share} and~\ref{lemma:pi-cyclic} prove uniqueness up
to isomorphism.
\end{proof}

Given a simple module $S$, the unique projective indecomposable module of
which $S$ is a quotient is called the \demph{projective cover} of $S$.

\section{From projective indecomposable modules to simple modules}

In the last section, we showed that for every simple module $S$, there is a
unique projective indecomposable module $P$ for which there exists an
epimorphism $P \epic S$.  We now show that this process is bijective.  In
other words, we show that for every projective indecomposable module $P$,
there is a unique simple module $S$ for which there exists an epimorphism
$P \epic S$.

\begin{lemma}
\label{lemma:pi-maximal}
Every projective indecomposable module has exactly one maximal submodule.
\end{lemma}

\begin{proof}
Let $P$ be a projective indecomposable module.  By
Proposition~\ref{propn:proj-max}, $P$ has at least one maximal submodule.
Now let $M$ and $M'$ be maximal submodules, and consider the inclusions and
projections
\[
\xymatrix@R-5ex{
M \ar[rd]^\iota         &                       &P/M            \\
                        &
P \ar@{->>}[ru]^-\pi \ar@{->>}[rd]_-{\pi'}      &       \\
M' \ar[ru]_{\iota'}     &                       &P/M'.
}
\]
Since $P/M'$ is simple, $\im(\pi'\iota)$ is either $0$ or $P/M'$.  If $0$
then $M \sub \ker\pi' = M'$; but $M$ and $M'$ are maximal, so $M = M'$.  It
therefore suffices to prove that $\pi'\iota$ is not an epimorphism.
Suppose that it is.  Since $P$ is projective, there exists a homomorphism
$\psi$ such that
\[
\xymatrix{
                &P \ar@{.>}[ld]_\psi \ar@{->>}[rd]^{\pi'}       &       \\
M \ar[r]_\iota  &P \ar@{->>}[r]_-{\pi'}                         &P/M'
}
\]
commutes.  By Lemma~\ref{lemma:pi-cyclic}, $P$ is finitely generated, so by
Corollary~\ref{cor:indec-fitting}, the endomorphism $\iota\psi$ of $P$ is
nilpotent or invertible.  If nilpotent then $(\iota\psi)^n = 0$ for some $n
\geq 0$; but $\pi' = \pi'(\iota\psi)$, so $\pi' = \pi'(\iota\psi)^n = 0$,
contradicting the fact that $\pi'$ is an epimorphism to a nonzero module.
If invertible then $\iota$ is an epimorphism, so $M = P$, also a
contradiction.
\end{proof}

\begin{propn}
\label{propn:top}
For each projective indecomposable module $P$, there is a simple module,
unique up to isomorphism, that is a quotient of $P$.
\end{propn}

\begin{proof}
Immediate from Lemma~\ref{lemma:pi-maximal}.
\end{proof}

Given a projective indecomposable module $P$, the unique simple quotient of
$P$ is called the \demph{top} or \demph{head} of $P$.

\section{The bijection}

Assembling the results of the last two sections, we obtain our main
theorem.

\begin{thm}
\label{thm:main}
There is a bijection between the set of isomorphism classes of projective
indecomposable modules and the set of isomorphism classes of simple
modules, given by matching a projective indecomposable module $P$ with a
simple module $S$ if and only if there exists an epimorphism $P \epic S$.
\end{thm}

\begin{proof}
Immediate from Propositions~\ref{propn:proj-cover} and~\ref{propn:top}.
\end{proof}

\begin{example}
\label{eg:ut}
Let $A$ be the algebra of $2 \times 2$ upper triangular matrices
$\utm{a}{b}{c}$ over $K$.  The $A$-module $A$ has submodules
\[
P_1
=
\biggl\{ \utmd{a}{0}{0} \such a \in K \biggr\},
\qquad
P_2
=
\biggl\{ \utmd{0}{b}{c} \such b, c \in K \biggr\}
\]
satisfying $P_1 \oplus P_2 = A$.  By
Lemma~\ref{lemma:proj-basic}\bref{part:proj-summand-free}, $P_1$ and $P_2$ are
projective.  

Since $P_1$ is $1$-dimensional, it is simple, and in particular
indecomposable.  The existence of the identity homomorphism $P_1 \epic P_1$
implies that the simple module corresponding to the projective
indecomposable module $P_1$ is $P_1$ itself.

By an elementary calculation, $P_2$ has just one nontrivial submodule,
namely
\[
M 
=
\biggl\{ \utmd{0}{b}{0} \such b \in K \biggr\}.
\]
It follows that $P_2$ is indecomposable.  It also follows that $M$ is the
unique maximal submodule of $P_2$.  The simple module corresponding to the
projective indecomposable module $P_2$ is, therefore, $P_2/M$.  Explicitly,
$P_2/M$ is the vector space $K$ made into an $A$-module by the action
$\utm{x}{y}{z} \cdot c = zc$.

This example shows that a simple module need not be projective
indecomposable, or vice versa, as mentioned in Section~\ref{sec:basics}.
For if $P_2/M$ were projective then by
Lemma~\ref{lemma:proj-basic}\bref{part:proj-epi}, it would be a
$1$-dimensional direct summand of the indecomposable $2$-dimensional module
$P_2$.  Conversely, the projective indecomposable module $P_2$ is not
simple, having a nontrivial submodule $M$.

Lemma~\ref{lemma:pi-on-list} will imply that $P_1$ and $P_2$ are
the \emph{only} projective indecomposable $A$-modules, and, therefore, that
$P_1$ and $P_2/M$ are the only simple $A$-modules.
\end{example}

\section{The space of homomorphisms}

When a projective indecomposable module $P$ corresponds to a simple module
$S$ (that is, when there exists an epimorphism $P \epic S$), we can try to
describe the space $\Hom_A(P, S)$ of all homomorphisms $P \to S$.  This is
made easier by:

\begin{lemma}
\label{lemma:into-simple}
Every homomorphism into a simple module is either zero or an epimorphism.
\end{lemma}

\begin{proof}
The image of such a homomorphism is a submodule of the codomain $S$, and
is therefore $0$ or $S$.  
\end{proof}

In particular, this implies the following result, which can be compared to
Corollary~\ref{cor:indec-fitting} for indecomposable modules.

\begin{lemma}
\label{lemma:endo-simple}
Every endomorphism of a simple module is either zero or invertible.
\end{lemma}

\begin{proof}
Follows from Lemma~\ref{lemma:into-simple}, since a surjective endomorphism
of a finite-dimensional vector space is invertible.
\end{proof}

When $P$ and $S$ correspond as in Theorem~\ref{thm:main}, we can describe
$\Hom_A(P, S)$ in terms of $S$ alone: 

\begin{propn}
\label{propn:hom}
Let $P$ be a projective indecomposable module and $S$ a simple module.
Then $\Hom_A(P, S)$ is isomorphic as a vector space to either $0$ or 
$\End_A(S)$. 
\end{propn}

In the latter case, the isomorphism is \emph{not} canonical.

\begin{proof}
If $\Hom_A(P, S) \neq 0$ then we can choose a nonzero homomorphism $\pi
\from P \to S$.  By Lemma~\ref{lemma:into-simple}, $\pi$ is an epimorphism,
so by Lemma~\ref{lemma:pi-maximal}, $\ker\pi$ is the unique maximal
submodule of $P$.  Composition with $\pi$ defines a linear map
\[
- \of \pi \from \End_A(S) \to \Hom_A(P, S),
\]
which we will prove is an isomorphism.  It is injective, as $\pi$ is an
epimorphism.  To show that is surjective, let $\phi \in \Hom_A(P, S)$.  By
Lemma~\ref{lemma:into-simple}, $\phi$ is either $0$ or an epimorphism, so
$\ker\phi$ is either $P$ or a maximal submodule of $P$.  In either case,
$\ker\phi \supseteq \ker\pi$.  Hence $\phi$ factors through $\pi$, as
required.
\end{proof}

Although $\Hom_A(P, S)$ does not carry the structure of a $K$-algebra in any
immediately obvious way, $\End_A(S)$ does.  We now analyse that structure.

\begin{lemma}
\label{lemma:end-field}
Let $S$ be a simple module.  Then:
\begin{enumerate}
\item 
\label{part:end-skew}
the $K$-algebra $\End_A(S)$ is a skew field;

\item
\label{part:end-ac}
if $K$ is algebraically closed then the $K$-algebra $\End_A(S)$ is
canonically isomorphic to $K$. 
\end{enumerate}
\end{lemma}

\begin{proof}
Part~\bref{part:end-skew} is immediate from Lemma~\ref{lemma:endo-simple}.
For~\bref{part:end-ac}, we prove that the $K$-algebra homomorphism
\[
\begin{array}{ccc}
K       &\to            &\End_A(S)      \\
\lambda &\mapsto        &\lambda\cdot \id_S
\end{array}
\]
is an isomorphism.  It is injective, as $S$ is nonzero.  To prove
surjectivity, let $\theta \in \End_A(S)$.  Then $\theta$ is a linear
endomorphism of a nonzero finite-dimensional vector space over an
algebraically closed field, and so has an eigenvalue $\lambda$.  But
$\theta - \lambda \cdot \id_S$ is then a non-invertible $A$-endomorphism of
$S$, so by Lemma~\ref{lemma:endo-simple}, it must be zero.
\end{proof}

\begin{propn}
\label{propn:hom-ac}
Let $P$ be a projective indecomposable module and $S$ a simple module.
Suppose that $K$ is algebraically closed.  Then $\Hom_A(P, S)$ is
isomorphic as a vector space to either $0$ or $K$.
\end{propn}

\begin{proof}
Follows from Proposition~\ref{propn:hom} and
Lemma~\ref{lemma:end-field}\bref{part:end-ac}.
\end{proof}

In the latter case, the isomorphism $\Hom_A(P, S) \iso K$ is not canonical.

\section{Finitely many isomorphism classes}
\label{sec:ks}

We have shown that the set of isomorphism classes of projective
indecomposable modules is in bijection with the set of isomorphism classes
of simple modules.  Here we show that both sets are finite.  We give two
alternative proofs, each using a standard theorem whose proof we omit.

The first uses the Jordan-H\"older theorem (Theorem~3.11 of~\cite{CuRe} or
Theorem~1.1.4 of~\cite{Bens}) to show that there are only finitely many
simple modules.  A \demph{composition series} of a module $M$ is a chain
\begin{equation}        
\label{eq:comp-series}
0 = M_0 \subset M_1 \subset \cdots \subset M_{r - 1} \subset M_r = M
\end{equation}
of submodules in which each quotient $M_j/M_{j - 1}$ is simple.

\begin{thm}[Jordan--H\"older]
Every finitely generated module $M$ has a composition
series~\eqref{eq:comp-series}, and the modules $M_1/M_0, \ldots, M_r/M_{r -
  1}$ are independent of the composition series chosen, up to reordering
and isomorphism.
\end{thm}

These quotients $M_j/M_{j - 1}$ are called the \demph{composition factors}
of $M$.  Thus, every finitely generated module has a well-defined
set-with-multiplicity of composition factors, which are simple modules.  In
particular, this is true of the $A$-module $A$; write $S_1, \ldots, S_r$
for its composition factors.  (They need not all be distinct.) 

Whenever $N$ is a submodule of a finitely generated module $M$, the
composition factors of $M$ are the composition factors of $N$ together with
the composition factors of $M/N$, adding multiplicities.  Hence:

\begin{lemma}
\label{lemma:simple-on-list}
Every simple module is isomorphic to $S_j$ for some $j \in \{1, \ldots, r\}$.
\end{lemma}

\begin{proof}
Let $S$ be a simple module.  By Lemma~\ref{lemma:simple-cyclic}, $S$ is a
quotient of the $A$-module $A$.  Hence every composition factor of $S$ is a
composition factor $S_j$ of $A$.  But $S$ is simple, so its unique
composition factor is itself.
\end{proof}

Together with Theorem~\ref{thm:main}, this gives our first proof of:

\begin{propn}
\label{propn:finite-set}
There are only finitely many isomorphism classes of projective
indecomposable modules, and only finitely many isomorphism classes of
simple modules.
\qed
\end{propn}

For the second proof of Proposition~\ref{propn:finite-set}, we use the
Krull--Schmidt theorem (Theorem~6.12 of~\cite{CuRe} or Theorem~1.4.6
of~\cite{Bens}, for instance).

\begin{thm}[Krull--Schmidt]
Every finitely generated module is isomorphic to a finite direct sum $M_1
\oplus \cdots \oplus M_n$ of indecomposable modules, and $M_1, \ldots, M_n$
are unique up to reordering and isomorphism.
\end{thm}

In particular, the $A$-module $A$ is isomorphic to $P_1 \oplus \cdots
\oplus P_n$ for some indecomposable $A$-modules $P_i$, which are determined
uniquely up to order and isomorphism.  (They need not all be distinct.)  By
Lemma~\ref{lemma:proj-basic}\bref{part:proj-summand-free}, each $P_i$ is
projective.  Conversely:

\begin{lemma}
\label{lemma:pi-on-list}
Every projective indecomposable module is isomorphic to $P_i$ for some $i
\in \{1, \ldots, n\}$.
\end{lemma}

\begin{proof}
Let $P$ be a projective indecomposable module.  By
Lemma~\ref{lemma:pi-cyclic}, there is an epimorphism $A \epic P$, so by
Lemma~\ref{lemma:proj-basic}\bref{part:proj-epi}, $A \iso P \oplus Q$ for
some module $Q$.  Now $Q$ is a quotient of $A$ and therefore finitely
generated, so by the Krull--Schmidt theorem, $Q = Q_1 \oplus \cdots \oplus
Q_m$ for some indecomposable modules $Q_j$.  This gives $A \iso P \oplus
Q_1 \oplus \cdots \oplus Q_m$.  Each of the summands is indecomposable, so
by the uniqueness part of Krull--Schmidt, $P$ is isomorphic to some $P_i$.
\end{proof}

Together with Theorem~\ref{thm:main}, this provides a second proof of
Proposition~\ref{propn:finite-set}. 

\paragraph*{Acknowledgements} Thanks to Andrew Hubery for
telling me the argument in Section~\ref{sec:max}, and to Iain Gordon,
Alastair King and Michael Wemyss for helpful conversations.

\end{document}